\documentclass[12pt,final]{article}
\usepackage{amsmath,amsthm,amsfonts,amssymb,graphicx,enumerate,psfrag}
\usepackage{mathrsfs}
\usepackage{fullpage}

\usepackage[colorlinks,citecolor=blue,urlcolor=blue]{hyperref}
\usepackage[utf8]{inputenc} 
\newtheorem{lemma}{Lemma}
\newtheorem{theorem}{Theorem}
\newtheorem{proposition}{Proposition}
\newtheorem{corollary}{Corollary}
\newtheorem{remark}{Remark}
\newtheorem{question}{Question}
\newcommand{\dZ}{\mathbb {Z}}
\newcommand{\dR}{\mathbb {R}}
\newcommand{\tls}{{\rm t}_{\textsc{ls}}}
\newcommand{\tmls}{{\rm t}_{\textsc{mls}}}
\newcommand{\tp}{{\rm t}_{\textsc{rel}}}
\newcommand{\tmix}{{\rm t}_{\textsc{mix}}}
\newcommand{\EE}{{\mathbf{E}}}
\newcommand{\Var}{{\mathrm{Var}}}
\newcommand{\Ent}{{\mathrm{Ent}}}

\newcommand{\cE}{\mathcal{E}}
\newcommand{\cH}{\mathcal{H}}
\newcommand{\cX}{\mathcal {X}}
\title{Upgrading  MLSI to  LSI for reversible Markov chains}
\author{Justin Salez, Konstantin Tikhomirov and Pierre Youssef}
\begin{document}
\maketitle
\begin{abstract}
For  reversible Markov chains on  finite state spaces, we show that the modified log-Sobolev inequality (MLSI) can  be upgraded to a log-Sobolev inequality (LSI) at the surprisingly low  cost of degrading the associated constant by  $\log (1/p)$, where $p$ is the minimum non-zero transition probability.  We illustrate this by providing the first log-Sobolev estimate for Zero-Range processes  on arbitrary graphs. As another application, we determine the modified log-Sobolev constant of the Lamplighter chain on all bounded-degree graphs, and use it to provide negative answers to two open questions by Montenegro and Tetali (2006) and Hermon and Peres (2018).  Our proof builds upon the `regularization trick' recently introduced by the last two authors.
\end{abstract}
\tableofcontents
\section{Introduction}
Functional inequalities constitute a powerful set of tools for the study of the \emph{concentration of measure} phenomenon \cite{MR1767995,MR1849347}. On a conceptual level, those inequalities relate certain statistics of a measure (such as the variance and entropy) to the Dirichlet form of a Markov process that preserves the measure. Besides concentration, those inequalities are intimately connected to the rate of convergence to equilibrium of the considered Markov process. Perhaps the most popular functional inequalities are the \emph{Poincar\'e Inequality} (PI) and the  
\emph{log-Sobolev Inequality} (LSI), which were initially studied in the continuous setting.
On discrete spaces, a variant called the \emph{Modified log-Sobolev Inequality} (MLSI) has been put forward and exploited due to its connection with the entropic exponential ergodicity of the underlying semi-group. 
Informally, it is a standard fact  (see for example \cite{MR2283379}) that
$$
\text{LSI}\; \Rightarrow\;  \text{MLSI}\; \Rightarrow\; \text{PI},
$$
and the aim of this paper is to investigate the reverse implications.

\subsection{Setup}
 
Throughout the paper, we consider an irreducible Markov generator $Q$ on a finite state space $\cX$, and we assume that $Q$ is reversible with respect to a probability measure $\pi$. 
The associated \emph{Dirichlet form} is given, for any observables $f,g \colon\cX\to\dR$, by the formula
\begin{eqnarray*}
\cE\left(f,g\right) & := & \frac{1}{2}\sum_{x,y\in\cX}\pi(x)Q(x,y)\left(f(x)-f(y)\right)\left(g(x)-g(y)\right).
\end{eqnarray*}
The \emph{variance} and \emph{entropy} of a function $f\colon\cX\to(0,\infty)$ are defined as $\Var(f):=\EE[f^2]-\EE^2[f]$ and $
\Ent(f) := \EE[f\log f]-\EE[f]\log\EE[f]$, where  $\EE[\cdot]$ denotes expectation with respect to $\pi$ and where `$\log$' stands for the natural logarithm. The  \emph{log-Sobolev} constant $\tls$, the \emph{modified log-Sobolev} constant $\tmls$, and the \emph{Poincar\'e} constant $\tp$ are then respectively defined as the optimal constants in the functional inequalities
\[
\label{LSI}
\tag{LSI}
\Ent(f) \ \le \ t\cE(\sqrt{f},\sqrt{f});
\]
\[
\label{MLSI}
\tag{MLSI}
\,\Ent(f) \ \le \ t\cE\left(f,\log f\right);
\]
\[
\label{Poincare}
\tag{PI}
\,\Var(f) \ \le \ t\cE\left(f,f\right),
\]
for all functions $f\colon\cX\to(0,\infty)$. We refer the unfamiliar reader to the seminal papers \cite{MR1410112,MR1796718,MR2283379}, the  book \cite{MR2341319}, or the more recent works \cite{MR3773799,MR4243518,MR4372142,MR4414692} for details about those fundamental constants and their many variants. 

It is classical that (\ref{LSI}), (\ref{MLSI}), (\ref{Poincare}) are decreasing in strength, in the exact sense that
\begin{eqnarray}
\label{comparison}
\frac{\tp}{2} \ \leq & \tmls & \leq \ \frac{\tls}{4}. 
\end{eqnarray}
The leftmost inequality is obtained by a standard perturbation argument around constant functions, while the rightmost one is a direct consequence of the inequality
\begin{eqnarray}
\label{LSIvsMLSI}
4\cE\left(\sqrt{f},\sqrt{f}\right) & \le & \cE(f,\log{f}),
\end{eqnarray}
for all $f\colon\cX\to(0,\infty)$, which in turns follows from the easy bound $4(\sqrt{u}-\sqrt{v})^2\le  (u-v)\log \frac uv$ for all $u,v>0$.  However, in contrast with  what happens in the smooth setting of  diffusions on  manifolds, the functional inequality (\ref{LSIvsMLSI}) can here \emph{not} be reversed uniformly in $f$: indeed, as soon as $|\cX|\ge 2$, the quantity $\cE\left(f,\log{f}\right)$ can be made arbitrarily larger than $\cE\left(\sqrt{f},\sqrt{f}\right)$ by simply increasing the value of $f$ at a single point. This degeneracy is  one of the many infamous consequences of the  lack of a \emph{chain rule} for Markov generators on discrete spaces \cite{MR1767995}. It introduces a fundamental discrepancy between the inequalities (\ref{MLSI}) and (\ref{LSI}), which results in two very different meanings for their optimal constants: $\tmls$  measures the \emph{entropy production} along the semi-group, while  $\tls$  quantifies the much stronger   \emph{hyper-contractivity}. As a result, the ratio $\tls/\tmls$ can be arbitrarily large even for two-point chains (see \cite[Section 3]{MR2283379}), and any general  upper bound on it will inevitably have to depend on the data $(Q,\pi)$. Understanding \emph{how} is the task we undertake here. 

To the best of our knowledge, the only  general upper bound available on the ratio $\tls/\tmls$  is the one resulting from the chain of inequalities
\begin{eqnarray}
\label{poor}
\frac{\tls}{\tmls} &  \le & \frac{2\tls}{\tp} \ \le\ \frac{2}{1-2\pi_\star}\log\left(\frac{1}{\pi_\star}-1\right),
\end{eqnarray}
 where $\pi_{\star}:= \min_{x\in \cX} \pi(x)$. The first inequality follows from (\ref{comparison}), and the second from a direct comparison with the Dirichlet form of the trivial chain $Q(x,y)=\pi(y)$, whose log-Sobolev constant is explicit (see, e.g. \cite{MR1410112}). Note that this `extreme' chain, which mixes in a single jump, actually saturates both inequalities in (\ref{poor}).  However, a quick look at more reasonable examples should convince the reader that the  ratio ${\tls}/{\tmls}$ is typically much smaller than  $\log\left({1}/{\pi_\star}\right)$ in practice, and this observation was the motivation behind the present paper.

\subsection{Main result}
Writing $
E  :=  \left\{(x,y)\in\cX^2\colon Q(x,y)>0\right\}
$ for the (symmetric) set of allowed transitions and  $Q(x):=\sum_{y\in\cX\setminus\{x\}}Q(x,y)$ for the total jump rate at $x\in \cX$, we define the \emph{sparsity} parameter
\begin{eqnarray}
\label{def:p}
p & := &  \min_{(x,y)\in E}\frac{Q(x,y)}{Q(x)\vee Q(y,x)} .
\end{eqnarray}In particular,  in the standard \emph{stochastic case} where $Q(x)=1$ for all $x\in\cX$, this reduces to
\begin{eqnarray*}
p & = & \min_{(x,y)\in E}Q(x,y).
\end{eqnarray*}
Our main result is that the lost equivalence between  (\ref{MLSI}) and  (\ref{LSI}) on discrete spaces can be restored  at the surprisingly low  cost $\log(1/p)$ only.
\begin{theorem}[Upgrading MLSI to LSI]\label{th:main}For any reversible Markov generator, we have 
\begin{eqnarray*}
\tls & \le & 20\, \tmls \log\left(\frac{1}{p}\right).
\end{eqnarray*}
\end{theorem}
The improvement over (\ref{poor}) can be considerable, since $p$ is typically much larger than $\pi_\star$. To appreciate this, consider the important case of simple random walk on a finite graph $G=(V_G,E_G)$: the classical estimate (\ref{poor}) predicts $\tls/\tmls=O(\log |V_G|)$, whereas our result gives $\tls/\tmls=O(\log d)$ where $d$ denotes the maximum degree. As a concrete example, when $G$ is the transposition walk on the symmetric group $\mathfrak S_n$, we obtain $\tls/\tmls=O(\log n)$ instead of $\tls/\tmls=O(n\log n)$. Incidentally, this example demonstrates that our result is sharp except for the value of the universal prefactor, which we did not try to optimize (see, e.g. \cite{DS}). On bounded-degree graphs, Theorem \ref{th:main}  shows that (\ref{MLSI}) and (\ref{LSI}) are actually equivalent, a new fact with surprising consequences (see Section \ref{sec:lamplighter} below).

Beyond the theoretical interest of a universal comparison between  (\ref{MLSI}) and (\ref{LSI}), Theorem \ref{th:main} can be used in practice to derive new, ready-to-use  functional-analytic estimates. Obviously, this can be done in two different directions. First, we can produce lower bounds on $\tmls$ in situations where a lower bound on $\tls$ is available. In Section \ref{sec:lamplighter} below, we illustrate this by determining the modified log-Sobolev constant of the Lamplighter chain on all bounded-degree graphs, and we use this to provide negative answers to two open questions by Montenegro and Tetali  and by Hermon and Peres. Conversely, there are several concrete classes of  chains for which a (\ref{MLSI})  was   established by methods that do not carry over to (\ref{LSI}). Examples include Bernoulli-Laplace models \cite{MR2548501}, Zero-Range processes \cite{MR4332696,MR2548501}, or negatively-dependent measures on the Boolean hypercube \cite{2019arXiv190202775H,MR4203344}. In such examples, using Theorem \ref{th:main} to convert the known upper bound on $\tmls$ into a new one on $\tls$ is of practical interest for at least three reasons:
\begin{enumerate}[(i)]
\item \emph{Mixing times}: writing $f_t$ for the density  with respect to equilibrium of the Markov process at time $t$, the parameter $\tmls$ only controls the decay rate of the relative entropy $\Ent(f_t)$, while $\tls$ controls  the much stronger uniform norm $\left\|f_t-1\right\|_\infty$  (see, e.g., \cite{MR2341319}). 
\item \emph{Isoperimetry}: by specializing (\ref{LSI}) to $\{0,1\}-$valued functions, the constant $\tls$ captures   \emph{small-set expansion}, which has numerous applications (see, e.g.,  \cite{10.1214/22-EJP749}).
\item \emph{Robustness}: unlike $\tmls$, an estimate on $\tls$ can  be transferred to other chains using the classical and well-developed comparison theory for Dirichlet forms (see, e.g., \cite{MR1233621}). 
\end{enumerate}
In Section \ref{sec:ZRP} below, we will illustrate this by providing the very first log-Sobolev estimate  for the Zero-Range Process with increasing rates on arbitrary graphs.
\subsection{Application to Lamplighter chains}\label{sec:lamplighter}
Fix a finite graph $G=(V_G,E_G)$ and imagine that each vertex is equipped with a lamp that can be either \emph{off} or \emph{on}. Now, consider a lamplighter performing a simple random walk on $G$ and randomly switching the lamps off or on on his way. More formally, let us describe the state of the system by a pair $(i,\sigma)$ where $i\in V$ represents the position of the lamplighter and $\sigma\in\{0,1\}^V$ indicates which lamps are on. The \emph{Lamplighter chain} on $G$ is the continuous-time Markov chain on $\cX:=V\times\{0,1\}^V$ whose generator acts on any  $f\colon\cX\to\dR$ as follows:
\begin{eqnarray*}
\label{def:lamplighter}
(Qf)(i,\sigma) & = & \frac{1}{2}\left(f(i,\sigma^i)-f(i,\sigma)\right)+\frac{1}{2\deg(i)}\sum_{j\sim i}\left(f(j,\sigma)-f(i,\sigma)\right).
\end{eqnarray*}
In this formula, the notation $j\sim i$ means that $\{i,j\}\in E_G$, while $\sigma^i$ denotes the configuration obtained from $\sigma\in \{0,1\}^V$ by replacing $\sigma_i$ with $1-\sigma_i$. The behavior of this process is extremely rich and has drawn considerable interest across various fields, including graph theory, group theory, spectral theory, discrete functional analysis and probability. We refer the reader to the book \cite[Chapter 19]{MR3726904} and the references therein for a quick introduction. 

The works \cite{MR1449833,MR2110019,MR2996735,MR3269988,TENT} are all devoted to the fundamental question of relating the  mixing properties of the Lamplighter chain  to those of the underlying graph $G$. In particular, the relaxation time, the total-variation mixing time and the uniform mixing time are now completely understood. This is also true for the log-Sobolev constant, at least on bounded-degree graphs. Specifically, the Lamplighter chain on $G$ was shown in \cite{MR2996735} to satisfy
\begin{eqnarray}
\label{LSI:lamplighter}
\tls & \asymp_d & \frac{|V_G|}{\gamma(G)},
\end{eqnarray}
where $\gamma(G)$ denotes the spectral gap of $G$ and $d=\max_{i\in V_G}\deg(i)$  the maximum degree. Here, the notation $a\asymp_d b$ means that the ratio $a/b$ is bounded from above and below by positive numbers that depend only on $d$. However,  in contrast with many other mixing parameters, nothing seems to be known about the modified log-Sobolev constant of Lamplighter chains, even on simple graphs such as the $n-$cycle $\dZ_n$. Since the Lamplighter chain has sparsity  $p=1/(2d)$,  our main result allows us to determine  $\tmls$ on all bounded-degree graphs.
\begin{corollary}[MLSI for the Lamplighter chain on bounded-degree graphs] We have
\label{co:lamplighter}
\begin{eqnarray*}
\tmls & \asymp_d & \frac{|V_G|}{\gamma(G)}.
\end{eqnarray*}
\end{corollary}
On the discrete torus $G=\dZ^d_n$, we thus obtain $\tmls=\Theta(n^{d+2})$  for any fixed dimension $d\ge 1$. Surprisingly, this is always much larger than the {total-variation mixing time} of the Lamplighter chain, which is known \cite{MR1449833,MR2110019} to be $\tmix=\Theta(n^{2})$ when $d=1$, $\tmix=\Theta(n^{2}\log^2 n)$ when $d=2$, and $\tmix=\Theta(n^{d}\log n)$ for any fixed $d\ge 3$. Thus, Lamplighter chains on discrete tori constitute a simple  family of counter-examples to the following classical question, which appears as Question 8.2 in the classical monograph \cite{MR2341319} by Montenegro and Tetali.
\begin{question}[Question 8.2 in \cite{MR2341319}]Is there a universal constant $c>0$ such that 
$
\tmls  \le  c\,\tmix,
$
for all irreducible Markov chains?
\end{question}
Very recently,  Hermon and Peres \cite[Question 7.1]{MR3773799} proposed the following more reasonable conjecture, in which  $\tmix$ is replaced by the larger {relative-entropy mixing time} ${\rm t}_{\textsc{ent}}$ (see \cite{TENT} or \cite{MR3773799} for the precise definition).
\begin{question}[Question 7.1 in \cite{MR3773799}]Is there a universal constant $c>0$ such that 
$
\tmls  \le  c\,{\rm t}_{\textsc{ent}},
$
for all irreducible Markov chains?
\end{question}
Since the Lamplighter chain on the $n-$cycle $\dZ_n$ is known to satisfy ${\rm t}_\textsc{ent} =\Theta(n^2\log n)$  \cite{TENT},  the estimate $\tmls=\Theta(n^3)$ given by Corollary \ref{co:lamplighter} again provides a negative answer to this  question. Finally, we note that our example also refutes Question 7.2 in the same paper, which asks for an even stronger upper bound on $\tmls$.

\subsection{Application to Zero-Range Processes}
\label{sec:ZRP}
Introduced by Spitzer \cite{Spitzer}, the \emph{Zero-Range Process} (\textsc{ZRP}) is a  generic interacting particle system  in which individual jumps occur at a rate that only depends on the current number of particles present at the source. The model is parameterized by the following ingredients:
\begin{itemize}
\item two  integers $m,n\ge 1$ representing the number of particles and sites, respectively;
\item a symmetric stochastic matrix $G=(G_{ij})_{1\le i,j\le n}$  specifying the geometry;
\item a function $r_i\colon\{1,2,\ldots\}\to(0,\infty)$ encoding the kinetics at each site $i\in [n]$.
\end{itemize}
The \textsc{ZRP} with these parameters is a continuous-time Markov chain on the state space 
\begin{eqnarray}
\cX & := & \left\{x=(x_1,\ldots,x_n)\in\dZ_+^n\colon \sum_{i=1}^n x_i=m\right\},
\end{eqnarray}
where $x_i$ represents the number of particles at site $i$. The action of the generator is given by 
\begin{eqnarray}
\label{def:markov}
(Q f)(x) & := & \sum_{1\le i,j\le n} r_i(x_i)G_{ij}\left(f(x+\delta_j-\delta_i)-f(x)\right),
\end{eqnarray}
where $(\delta_1,\ldots,\delta_n)$ denotes the canonical $n-$dimensional basis, and with the convention that $r_i(0)=0$ for all $i\in [n]$ (no jumps from empty sites). In words, a site $i$ with $k\ge 1$ particles expels a particle at rate $r_i(k)$, and the destination is chosen according to  $G$. A natural choice for the latter is the transition matrix of simple random walk on a regular graph. In fact, the model is already interesting on the complete graph (the so-called \emph{mean-field} case):
\begin{eqnarray}
\label{def:MF}
\forall (i,j)\in[n]^2,\qquad G_{ij} &  = & \frac{1}{n}.
 \end{eqnarray}
Obtaining quantitative estimates on the convergence to equilibrium of the ZRP has been and continues to be a subject of active research, see e.g., \cite{Cap,MR4089493} in the mean-field case and \cite{MR2184099,MR2147317} on integer lattices. A standard assumption on the rate functions $(r_i)_{1\le i\le n}$ is that their increments all lie in a fixed  compact subset of $(0,\infty)$:
\begin{eqnarray}
\label{assume:rates}
\delta \  \le & r_i(k+1)-r_i(k) & \le  \ \Delta,
\end{eqnarray}
for some $\delta,\Delta>0$ and every $i\in [n]$ and  $k\in\dZ_+$.
Under this \emph{weak interaction} condition, and in the \emph{mean-field} setting (\ref{def:MF}), the ZRP was shown in \cite{MR4332696} to satisfy  the dimension-free MLSI
\begin{eqnarray*}
\tmls & \le & \frac{2\Delta}{\delta^2}.
\end{eqnarray*}
Note that our sparsity parameter $p$ here satisfies $p\ge \frac{\delta}{\Delta mn}$, because $Q(x)\le \sum_{i=1}^nr_i(x_i)\le \Delta m$ for all $x\in\cX$ and $Q(x,y) \ge \delta/n$ for all $(x,y)\in E$.  
Thus,  Theorem \ref{th:main} produces the following estimate which, to the best of our knowledge, is the very first LSI for the mean-field ZRP.
\begin{corollary}[LSI for mean-field ZRP]Under assumptions (\ref{def:MF}) and (\ref{assume:rates}), we have
\begin{eqnarray*}
\label{MFZRP:LSI}
\tls & \le & \frac{40\Delta}{\delta^2}\log\left(\frac{\Delta mn}{\delta}\right),
\end{eqnarray*}
for any choice of the dimension parameters $n,m$.
\end{corollary}
The dependency in $n$ is optimal, as can be seen by investigating the case of a single particle ($m=1$). As already mentioned, one of the advantages of $\tls$ over $\tmls$ is its robustness under comparison of Dirichlet forms. This is particularly true in the present setting, because it was shown  in \cite{MR3984254} that replacing  a general symmetric stochastic matrix $G$ by its mean-field version (\ref{def:MF}) can not increase the Dirichlet form of the ZRP by more than a factor $1/\gamma(G)$, where $\gamma(G)$ denotes the  spectral gap of  $G$. Consequently, our mean-field LSI estimate   can be directly transferred to arbitrary geometries, yielding the following result. 
To the best of our knowledge, log-Sobolev estimates  for the ZRP were so far restricted to lattices \cite{MR2184099}.  
\begin{corollary}[LSI for ZRP on arbitrary geometries]\label{co:ZRP}Under Assumption (\ref{assume:rates}), we have
\begin{eqnarray*}
\tls & \le & \frac{40\Delta}{\gamma(G)\delta^2}\log\left(\frac{\Delta m n}{\delta}\right),
\end{eqnarray*}
for any choice of the dimension parameters $n,m$ and of the symmetric stochastic matrix $G$. 
\end{corollary}

\begin{remark}[Asymmetric geometries]Interestingly, the symmetry of  $G$ is  never actually used in \cite{MR4332696,MR3984254}, so Corollary \ref{co:ZRP} generalizes to any  stochastic matrix $G$ as follows:
\begin{eqnarray*}
\tls & \le & \frac{40\Delta}{\gamma(G)\delta^2}\log\left(\frac{\Delta m }{\delta p_\star}\right),
\end{eqnarray*}
where $p_\star$ denotes the smallest entry of the invariant probability vector of $G$, and where $\gamma(G)$  denotes the spectral gap of the additive reversibilization of $G$.
\end{remark}
\section*{Acknowledgement.} The three authors warmly thank Jonathan Hermon for pointing out the open problem \cite[Question 7.1]{MR3773799}. J.S. was partially supported by Institut Universitaire de France, and K.T. was partially supported by the Sloan Research Fellowship and by the NSF grant DMS 2054666. 
\section{Proof of the main result}

The proof of Theorem \ref{th:main} builds upon a simple but fruitful idea, recently introduced by the last two authors in \cite{TY}: the \emph{regularization trick}. It consists in restricting functional inequalities to observables $f\colon\cX\to(0,\infty)$ that are  \emph{smooth}, in  an appropriate sense.
\subsection{The regularization trick}
Fix a parameter $r\ge 1$. Following \cite{TY}, we say that a function $f\colon\cX\to(0,\infty)$ is 
$r-$\emph{regular} if 
\begin{eqnarray}
\label{assume:reg}
\forall (x,y)\in E, \qquad f(x) & \le & rf(y).
\end{eqnarray}
Our starting point is the elementary but new observation that, once restricted to $r-$regular functions, the basic  inequality (\ref{LSIvsMLSI}) can be reversed at the optimal cost
\begin{eqnarray}
\label{def:H}
\cH(r) & := & \frac{\sqrt{r}+1}{\sqrt{r}-1}\log r.
\end{eqnarray}
Note that $\cH(\infty)=\infty$, in agreement with our earlier observation that the unrestricted inequality (\ref{LSIvsMLSI}) can \emph{not} be reversed. At the other extreme, we have $\cH(1+)=4$, in agreement with the fact that (\ref{LSIvsMLSI}) is an equality in the infinitely smooth setting of diffusions on manifolds. 

\begin{lemma}[Exploiting regularity]\label{lm:r} If $f\colon\cX\to(0,\infty)$ is $r-$regular, then 
\begin{eqnarray*}
 \cE(f,\log f) & \le & \cH(r)\cE(\sqrt{f},\sqrt{f}).
\end{eqnarray*}
\end{lemma}
\begin{proof}
With  $\cH(1):=4$, the formula (\ref{def:H}) defines a  function $\cH\colon (0,\infty)\to (0,\infty)$  which increases on $[1,\infty)$ and satisfies  $\cH(u)=\cH(u^{-1})$ for all $u\in(0,\infty)$. It follows that $\cH(u) \le \cH(r)$ for all $u\in[r^{-1},r]$. On the other hand, elementary manipulations give
\begin{eqnarray*}
(u-v)\log\left(\frac{u}{v}\right) & = & \cH\left(\frac{u}{v}\right)(\sqrt{u}-\sqrt{v})^2,
\end{eqnarray*}
for all $u,v>0$. Taking $u=f(x)$ and $v=f(y)$, and recalling our assumption (\ref{assume:reg}), we  obtain
\begin{eqnarray*}
(f(x)-f(y))\log \left(\frac{f(x)}{f(y)}\right) & \le & \cH\left(r\right)(\sqrt{f(x)}-\sqrt{f(y)})^2,
\end{eqnarray*}
for any $(x,y)\in E$. To conclude, we multiply by $\pi(x)Q(x,y)$  and sum over all $x,y\in\cX$. 
\end{proof}
Let us respectively denote by $\tls(r)$ and $\tmls(r)$ the optimal constants in the inequalities (\ref{LSI}) and (\ref{MLSI}), when restricted to $r-$regular functions. The above lemma readily implies
\begin{eqnarray}
 \tls(r) & \le & \cH(r)\tmls(r),
\end{eqnarray}
which can be seen as a regularized version of Theorem \ref{th:main}. To conclude,  we now need to relate  the constants $\tmls(r)$ and $\tls(r)$ to their unregularized versions $\tmls$ and $\tls$. Of course,  we trivially have
$
\tmls(r)  \le  \tmls$
and $\tls(r)  \le  \tls$, by definition. 
We will now show that those inequalities can be reversed, provided $r$ is  large enough. Specifically, we henceforth set
\begin{eqnarray}
\label{def:r}
r & := & \frac{4}{p^{2}},
\end{eqnarray} 
and we assume that $p\le 1/2$ (if this is not the case, then irreducibility and reversibility imply that $|\cX|\le 2$, so that Theorem \ref{th:main} can be checked by hands). Note that we then have   $3\cH(r)\le {20}\log \frac{1}{p}$. Thus, Theorem \ref{th:main} is a consequence of the following crucial result.
\begin{theorem}[Regularization]\label{th:reg}With $r$ as in (\ref{def:r}), we have
$
\tls  \le 3\tls(r)$  and $\tmls  \le  3\tmls(r)$.

\end{theorem}
In other words, to establish (\ref{LSI}) or (\ref{MLSI}) for a reversible Markov chain, it is enough to restrict  attention to $r-$regular observables. Such a costless regularization is of course interesting beyond its role in the present paper. A first version of it (for $\tmls$, and with an additional factor $\gamma=\frac{\max\pi}{\min\pi}$ in the regularization parameter $r$) was recently used  in \cite{TY} to obtain a sharp (\ref{MLSI}) for the \emph{switch chain} on regular bipartite graphs and as a consequence prove a long-standing conjecture about the mixing time of that chain (see \cite{TY'} and references therein).  Our proof below follows the same strategy, but optimizes it so as to remove the dependency of $r$ on $\gamma$. Note that this improvement  is crucial for our application to the ZRP.

We write $d(\cdot,\cdot)$ for the graph distance on $(\cX,E)$. Following \cite{TY}, we   fix  an arbitrary observable $f\colon\cX\to(0,\infty)$ and we define a new observable $f_\star\colon\cX\to(0,\infty)$ by
\begin{eqnarray*}
 \forall x\in\cX,\qquad f_\star(x) & := & \max_{z\in\cX}\, r^{-d(x,z)} f(z).
 \end{eqnarray*} 
It is immediate to check that $f_\star$ is $r-$regular and above $f$ (it is in fact the smallest such function, but we will not use this property here). The following propositions  show that the quantities $\cE(\sqrt{f},\sqrt{f})$, $\cE(f,\log f)$ and $\Ent(f)$ do not change much upon replacing $f$ by $f_\star$. 
\begin{proposition}[Comparison of Dirichlet forms]\label{pr:dirichlet} For $r$ as in (\ref{def:r}),  we have 
\begin{eqnarray*}
\cE(\sqrt{f_\star},\sqrt{f_\star})\ \le\ \frac 43 \cE(\sqrt{f},\sqrt{f}) & \textrm{ and } & \cE({f_\star},\log {f_\star})\ \le\ \frac 43 \cE(f,\log {f}).
\end{eqnarray*}
\end{proposition}
\begin{proposition}[Comparison of entropies]\label{pr:entropy} For $r$ as in (\ref{def:r}), we have 
$\Ent(f)\le 2\Ent(f_\star)$.
\end{proposition}
Those propositions clearly imply Theorem \ref{th:reg}, and we henceforth focus on their proofs.
\subsection{Comparing Dirichlet forms (Proposition \ref{pr:dirichlet})}

The proof of Proposition \ref{pr:dirichlet} consists of three steps, of which only the last one uses the specific choice  for $r$ made at (\ref{def:r}). We  define the distance from an edge $e\in E$ to  $e'\in E$ in the obvious way: $d(e,e')=\ell-1$, where $\ell\ge 1$ is the minimum length of a path $(x_0,\ldots,x_\ell)$ with $(x_0,x_1)=e$ and $(x_{\ell-1},x_\ell)=e'$.
We first compare the  variations of $f$ and $f_\star$ across edges. 
\begin{lemma}\label{lm:local}For each $e=(x,y)\in E$ with $f_\star(x)\le f_\star(y)$, there is  $e'=(x',y')\in E$ so that 
\begin{eqnarray}
\label{goal}
r^{-d(e,e')} f(x') \ \le & f_\star(x) \ \le \  f_\star(y) & \le \ r^{-d(e,e')} f(y').
\end{eqnarray}
\end{lemma}
\begin{proof}If $f_\star(y)=f(y)$, then we simply choose $e'=e$ and (\ref{goal}) trivially holds. If on the contrary $f_\star(y)<f(y)$, then we take $e'=(x',y')$ where $y'$ is such that $f_\star(y)=r^{-d(y,y')}f(y')$, and where $x'$ is the penultimate vertex on a geodesic from $y$ to $y'$. We then have $d(e,e')=d(y,y')= 1+d(y,x')\ge d(x,x')$. Thus, the last inequality in (\ref{goal}) holds with equality, while the definition of $f_\star$ ensures that  $f_\star(x)\ge r^{-d(x,x')} f(x')\ge r^{-d(e,e')} f(x')$.
\end{proof}
We  now use this local comparison to estimate how the global quantities $\cE(\sqrt{f},\sqrt{f})$ and $\cE(f,\log {f})$  change  upon  replacing $f$ with $f_\star$.  For $(x,y)\in E$, we introduce the short-hand
\begin{eqnarray*}
c(x,y) & := & \pi(x)Q(x,y).
\end{eqnarray*}

\begin{lemma}We have 
$
\cE(\sqrt{f_\star},\sqrt{f_\star})  \le \kappa\,  \cE(\sqrt{f},\sqrt{f})$ and $
\cE(f_\star,\log f_\star)  \le  \kappa\, \cE(f,\log f)$, where
\begin{eqnarray}
\label{def:kappa}
\kappa & := & \max_{e'\in E}\left\{\frac{1}{c(e')}\sum_{e\in E}c(e)r^{-d(e,e')}\right\}.
\end{eqnarray}
\end{lemma}
\begin{proof} Let $\nabla f(e) :=  f(y)-f(x)$ denote the discrete gradient of $f$ across an edge $e=(x,y)$. To each edge  $e\in E$, Lemma \ref{lm:local} associates a new edge $e'\in E$ such that
\begin{eqnarray*}
\left(\nabla \sqrt{f_\star}\right)_+ (e)
& \le & r^{-d(e,e')}\left(\nabla \sqrt{f}\right)_+ (e'),
\end{eqnarray*}
where $h_+=\max(h,0)$ denotes the positive part of $h$. Consequently, we have 
\begin{eqnarray*}
\cE(\sqrt{f_\star},\sqrt{f_\star}) & = & \sum_{e\in E}c(e)\left(\nabla \sqrt{f_\star}\right)^2_+(e) \\
& \le & \sum_{e,e'\in E}c(e)r^{-d(e,e')}\left(\nabla \sqrt{f}\right)^2_+(e')\\
& \le &  \kappa\sum_{e'\in E}c(e')\left(\nabla \sqrt{f}\right)^2_+(e') \ = \ \kappa \cE(\sqrt{f},\sqrt{f}),
\end{eqnarray*}
as desired. The second claim is obtained in exactly the same way. 
\end{proof}
To obtain Proposition \ref{pr:dirichlet}, it finally remains to estimate the constant $\kappa$ defined at (\ref{def:kappa}).
\begin{lemma}\label{lm:kappa}Choosing $r$ as in (\ref{def:r}) ensures that $\kappa\le 4/3$. 
\end{lemma}
\begin{proof}If $e,e'\in E$ satisfy $d(e,e')=1$, then we can write $e=(x,y)$ and $e'=(y,z)$. Using reversibility and the definition of $p$ at (\ref{def:p}), we have
\begin{eqnarray*}
c(e) & = & \pi(x)Q(x,y) \ = \  \pi(y)Q(y,x) \ \le \ \pi(y)Q(y)\ \le \ p^{-1}\pi(y)Q(y,z) \ = \ p^{-1}c(e').
\end{eqnarray*}
By an immediate induction, we deduce that $c(e)\le  p^{-d(e,e')}c(e')$ for all $e,e'\in E$. Thus,
\begin{eqnarray*}
\kappa & \le & \max_{e'\in E}\left\{\sum_{e\in E}(rp)^{-d(e,e')}\right\} \ \le \  \sum_{k=0}^\infty\Delta^k(rp)^{-k},
\end{eqnarray*}
where $\Delta$ denotes the maximum degree of the graph $(\cX,E)$. But $\Delta\le p^{-1}$,  because  from  every vertex, the outgoing transition probabilities  are at least $p$ and must sum to $1$. Thus, $\kappa\le \sum_k (p^2r)^{-k}=4/3$, thanks to  our choice for $r$.
\end{proof}

\subsection{Comparing entropies (Proposition \ref{pr:entropy})}
As in \cite{TY}, we use the variational characterization of entropy \cite{MR3185193}: for any $g\colon\cX\to(0,\infty)$,
\begin{eqnarray}
\label{varentropy}
\Ent(g) & = & \max\left\{\EE[gh]\colon h\in\dR^{\cX},  \EE[e^h]\le 1\right\}.
\end{eqnarray}
We start with an elementary lemma, which only uses the fact that $f_\star\ge f$.  
\begin{lemma}We have 
$
5\Ent(f)  \le  6\Ent(f_\star)+(6\log 6)\EE[f_\star-f]$.
\end{lemma}
\begin{proof}Upon multiplying $f$ (hence also $f_\star$) by a constant, we may assume without loss of generality that $\EE[f]=1$. Then, the function $h:=\log\left(\frac{1+5f}6\right)$ satisfies $\EE[e^h]=1$ so (\ref{varentropy}) yields
\begin{eqnarray*}
\Ent(f_\star) & \ge  & \EE[f_\star h] \ = \ \EE[fh]+\EE[(f_\star-f)h].
\end{eqnarray*}
The claim now readily follows from the pointwise bounds $h\ge \frac 56{\log f}$ and $h\ge \log \frac 16$.
\end{proof}
In view of this lemma, Proposition \ref{pr:entropy} boils down to the following result.
\begin{lemma}With $r$ as in (\ref{def:r}), we have $(3\log 6)\EE[f_\star-f] \le \Ent(f)$.
\end{lemma}
\begin{proof}For each $x\in\cX$, let us choose a state $T(x)\in\cX$ that achieves the maximum in the definition of $f_\star(x)$ (breaking ties arbitrarily), i.e.
\begin{eqnarray*}
f_\star(x) & = & r^{-d(x,T(x))}f\left(T(x)\right).
\end{eqnarray*}
Thus, $T(x)=x$ if and only if $f_\star(x)=f(x)$. Note that if we had $f(T(x))<f_\star(T(x))$, then 
\begin{eqnarray*}
f_\star(x) & < & r^{-d(x,T(x))}f_\star\left(T(x)\right) \ = \  r^{-d(x,T(x))}r^{-d(T(x),T^2(x))}f\left(T^2(x)\right)\ \le \  r^{-d(x,T^2(x))}f\left(T^2(x)\right),
\end{eqnarray*}
which would contradict the maximal definition of $f_\star(x)$. Thus, we must in fact have $f(T(x))=f_\star(T(x))$, or equivalently, $T^2(x)=T(x)$. This shows that  $T(A)=A^c$, where
\begin{eqnarray*}
A & := & \{x\in\cX\colon T(x)\ne x\} \ = \ \{x\in\cX\colon f(x)\ne f_\star(x)\}.
\end{eqnarray*}
Now, coming back to our goal, let us write
\begin{eqnarray*}
\EE[f_\star-f] & = & \sum_{x\in A}\pi(x)(f_\star(x)-f(x))\\
& = & \sum_{x\in A}\pi(x)f(T(x))r^{-d(x,T(x))}-\sum_{x\in A}\pi(x)f(x)\\
& = & \sum_{y\in T(A)}\pi(y)f(y)h(y)-\sum_{x\in A}\pi(x)f(x),
\end{eqnarray*}
where for $y\in T(A)$, we have introduced the short-hand
\begin{eqnarray*}
h(y) & := & \sum_{x\in T^{-1}(\{y\})}\frac{\pi(x)}{\pi(y)}r^{-d(x,y)}.
\end{eqnarray*}
Recalling that $T(A)=A^c$, we may set $h=-1$ on $A$ to rewrite the previous computation as  
\begin{eqnarray*}
 \EE[f_\star-f]  & = & \EE[fh].
\end{eqnarray*}
In light of the variational characterization of entropy (\ref{varentropy}) (with $3h\log 6$ instead of $h$), it remains to check that $\EE[6^{3h}]  \le  1$. As in the proof of Lemma \ref{lm:kappa}, our choice  $r=4p^{-2}$ easily ensures that $h\le \sum_{k=1}(p^2r)^{-k}= 1/3$, so that $6^{3h}\le 1+15h$ on $A^c$. Recalling the definition of $h$, we deduce that
 \begin{eqnarray*}
 \EE\left[(6^{3h}-1){\bf 1}_{A^c}\right] & \le & 15\EE[h{\bf 1}_{A^c}] \ = \ 
  15\sum_{x\in A}\pi(x)r^{-d(x,T(x))} \ \le \ \frac{15}{16}\pi(A),
 \end{eqnarray*}
because $r\ge 16$. On the other hand, $\EE\left[(6^{3h}-1){\bf 1}_{A}\right]=(6^{-3}-1)\pi(A)\le-\frac{15}{16}\pi(A)$.
\end{proof}

\bibliographystyle{plain}
\bibliography{draft}

\begin{thebibliography}{10}

\bibitem{MR2996735}
Evgeny Abakumov, Anne Beaulieu, Fran\c{c}ois Blanchard, Matthieu Fradelizi,
  Natha\"{e}l Gozlan, Bernard Host, Thiery Jeantheau, Magdalena Kobylanski,
  Guillaume Lecu\'{e}, Miguel Martinez, Mathieu Meyer, Marie-H\'{e}l\`ene
  Mourgues, Fr\'{e}d\'{e}ric Portal, Francis Ribaud, Cyril Roberto, Pascal
  Romon, Julien Roth, Paul-Marie Samson, Pierre Vandekerkhove, and Abdellah
  Youssfi.
\newblock The logarithmic {S}obolev constant of the lamplighter.
\newblock {\em J. Math. Anal. Appl.}, 399(2):576--585, 2013.

\bibitem{MR4372142}
Rados\text{ł}aw Adamczak, Bart\text{ł}omiej Polaczyk, and Micha\text{ł}
  Strzelecki.
\newblock Modified log-{S}obolev inequalities, {B}eckner inequalities and
  moment estimates.
\newblock {\em J. Funct. Anal.}, 282(7):Paper No. 109349, 76, 2022.

\bibitem{MR2283379}
Sergey~G. Bobkov and Prasad Tetali.
\newblock Modified logarithmic {S}obolev inequalities in discrete settings.
\newblock {\em J. Theoret. Probab.}, 19(2):289--336, 2006.

\bibitem{MR3185193}
St\'{e}phane Boucheron, G\'{a}bor Lugosi, and Pascal Massart.
\newblock {\em Concentration inequalities}.
\newblock Oxford University Press, Oxford, 2013.
\newblock A nonasymptotic theory of independence, With a foreword by Michel
  Ledoux.

\bibitem{MR2548501}
Pietro Caputo, Paolo Dai~Pra, and Gustavo Posta.
\newblock Convex entropy decay via the {B}ochner-{B}akry-{E}mery approach.
\newblock {\em Ann. Inst. Henri Poincar\'{e} Probab. Stat.}, 45(3):734--753,
  2009.

\bibitem{Cap}
Pietro Caputo and Gustavo Posta.
\newblock Entropy dissipation estimates in a zero-range dynamics.
\newblock {\em Probab. Theory Related Fields}, 139(1-2):65--87, 2007.
\newblock \href{http://www.ams.org/mathscinet-getitem?mr=MR2322692}{MR2322692}.

\bibitem{MR4414692}
Giovanni Conforti.
\newblock A probabilistic approach to convex ({$\phi$})-entropy decay for
  {M}arkov chains.
\newblock {\em Ann. Appl. Probab.}, 32(2):932--973, 2022.

\bibitem{MR4203344}
Mary Cryan, Heng Guo, and Giorgos Mousa.
\newblock Modified log-{S}obolev inequalities for strongly log-concave
  distributions.
\newblock {\em Ann. Probab.}, 49(1):506--525, 2021.

\bibitem{MR2184099}
Paolo Dai~Pra and Gustavo Posta.
\newblock Logarithmic {S}obolev inequality for zero-range dynamics.
\newblock {\em Ann. Probab.}, 33(6):2355--2401, 2005.
\newblock \href{http://www.ams.org/mathscinet-getitem?mr=MR2184099}{MR2184099}.

\bibitem{MR2147317}
Paolo Dai~Pra and Gustavo Posta.
\newblock Logarithmic {S}obolev inequality for zero-range dynamics:
  independence of the number of particles.
\newblock {\em Electron. J. Probab.}, 10:no. 15, 525--576, 2005.

\bibitem{MR1410112}
P.~Diaconis and L.~Saloff-Coste.
\newblock Logarithmic {S}obolev inequalities for finite {M}arkov chains.
\newblock {\em Ann. Appl. Probab.}, 6(3):695--750, 1996.

\bibitem{MR1233621}
Persi Diaconis and Laurent Saloff-Coste.
\newblock Comparison theorems for reversible {M}arkov chains.
\newblock {\em Ann. Appl. Probab.}, 3(3):696--730, 1993.

\bibitem{DS}
Persi Diaconis and Mehrdad Shahshahani.
\newblock Generating a random permutation with random transpositions.
\newblock {\em Z. Wahrsch. Verw. Gebiete}, 57(2):159--179, 1981.
\newblock \href{http://www.ams.org/mathscinet-getitem?mr=MR626813}{MR626813}.

\bibitem{10.1214/22-EJP749}
Yuval Filmus, Ryan O’Donnell, and Xinyu Wu.
\newblock {Log-Sobolev inequality for the multislice, with applications}.
\newblock {\em Electronic Journal of Probability}, 27(none):1 -- 30, 2022.

\bibitem{TENT}
Murali~K. Ganapathy and Prasad Tetali.
\newblock A tight bound for the lamplighter problem, 2006.

\bibitem{MR1449833}
Olle H\"{a}ggstr\"{o}m and Johan Jonasson.
\newblock Rates of convergence for lamplighter processes.
\newblock {\em Stochastic Process. Appl.}, 67(2):227--249, 1997.

\bibitem{MR3773799}
Jonathan Hermon and Yuval Peres.
\newblock A characterization of {$L_2$} mixing and hypercontractivity via
  hitting times and maximal inequalities.
\newblock {\em Probab. Theory Related Fields}, 170(3-4):769--800, 2018.

\bibitem{2019arXiv190202775H}
Jonathan {Hermon} and Justin {Salez}.
\newblock {Modified log-Sobolev inequalities for strong-Rayleigh measures}.
\newblock {\em arXiv e-prints}, page arXiv:1902.02775, Feb 2019.

\bibitem{MR3984254}
Jonathan Hermon and Justin Salez.
\newblock A version of {A}ldous' spectral-gap conjecture for the zero range
  process.
\newblock {\em Ann. Appl. Probab.}, 29(4):2217--2229, 2019.

\bibitem{MR4089493}
Jonathan Hermon and Justin Salez.
\newblock Cutoff for the mean-field zero-range process with bounded monotone
  rates.
\newblock {\em Ann. Probab.}, 48(2):742--759, 2020.

\bibitem{MR4332696}
Jonathan Hermon and Justin Salez.
\newblock Entropy dissipation estimates for inhomogeneous zero-range processes.
\newblock {\em Ann. Appl. Probab.}, 31(5):2275--2283, 2021.

\bibitem{MR3269988}
J\'{u}lia Komj\'{a}thy, Jason Miller, and Yuval Peres.
\newblock Uniform mixing time for random walk on lamplighter graphs.
\newblock {\em Ann. Inst. Henri Poincar\'{e} Probab. Stat.}, 50(4):1140--1160,
  2014.

\bibitem{MR1796718}
R.~Lata\text{ł}a and K.~Oleszkiewicz.
\newblock Between {S}obolev and {P}oincar\'{e}.
\newblock In {\em Geometric aspects of functional analysis}, volume 1745 of
  {\em Lecture Notes in Math.}, pages 147--168. Springer, Berlin, 2000.

\bibitem{MR1767995}
Michel Ledoux.
\newblock Concentration of measure and logarithmic {S}obolev inequalities.
\newblock In {\em S\'{e}minaire de {P}robabilit\'{e}s, {XXXIII}}, volume 1709
  of {\em Lecture Notes in Math.}, pages 120--216. Springer, Berlin, 1999.

\bibitem{MR1849347}
Michel Ledoux.
\newblock {\em The concentration of measure phenomenon}, volume~89 of {\em
  Mathematical Surveys and Monographs}.
\newblock American Mathematical Society, Providence, RI, 2001.

\bibitem{MR3726904}
David~A. Levin, Yuval Peres, and Elizabeth~L. Wilmer.
\newblock {\em Markov chains and mixing times}.
\newblock American Mathematical Society, Providence, RI, 2017.
\newblock Second edition of [ MR2466937], With a chapter on ``Coupling from the
  past'' by James G. Propp and David B. Wilson.

\bibitem{MR2341319}
Ravi Montenegro and Prasad Tetali.
\newblock Mathematical aspects of mixing times in {M}arkov chains.
\newblock {\em Found. Trends Theor. Comput. Sci.}, 1(3):x+121, 2006.
\newblock \href{http://www.ams.org/mathscinet-getitem?mr=MR2341319}{MR2341319}.

\bibitem{MR2110019}
Yuval Peres and David Revelle.
\newblock Mixing times for random walks on finite lamplighter groups.
\newblock {\em Electron. J. Probab.}, 9:no. 26, 825--845, 2004.

\bibitem{MR4243518}
Holger Sambale and Arthur Sinulis.
\newblock Modified log-{S}obolev inequalities and two-level concentration.
\newblock {\em ALEA Lat. Am. J. Probab. Math. Stat.}, 18(1):855--885, 2021.

\bibitem{Spitzer}
Frank Spitzer.
\newblock Interaction of {M}arkov processes.
\newblock {\em Advances in Math.}, 5:246--290 (1970), 1970.
\newblock \href{http://www.ams.org/mathscinet-getitem?mr=MR0268959}{MR0268959}.

\bibitem{TY}
Konstantin Tikhomirov and Pierre Youssef.
\newblock Regularized modified log-sobolev inequalities, and comparison of
  markov chains, 2022.

\bibitem{TY'}
Konstantin Tikhomirov and Pierre Youssef.
\newblock Sharp poincar\'e and log-sobolev inequalities for the switch chain on
  regular bipartite graphs.
\newblock {\em Probab. Theory Related Fields}, 2022.

\end{thebibliography}
\section*{Author affiliations}
Justin Salez: CEREMADE, Université Paris-Dauphine \& PSL, Place du Maréchal de Lattre de Tassigny, F-75775 Paris Cedex 16, FRANCE. 
\texttt{\small e-mail:  justin.salez@dauphine.psl.eu}\\

\noindent Konstantin Tikhomirov: School of Mathematics, Georgia Institute of Technology, 686 Cherry street, Atlanta, GA 30332, USA \& Department of Mathematical Sciences, Carnegie Mellon University, Wean Hall 6113, Pittsburgh, PA 15213, USA. \texttt{\small e-mail:  ktikhomi@andrew.cmu.edu}\\

\noindent Pierre Youssef: 
Division of Science, NYU Abu Dhabi, Saadiyat Island, Abu Dhabi, UAE \& Courant
Institute of Mathematical Sciences, New York University, 251 Mercer st, New York,
NY 10012, USA. 
\texttt{\small e-mail:  yp27@nyu.edu}
\end{document}